\documentclass{article}
\usepackage{amssymb}
\usepackage{amsthm}

\newtheorem{lemma}{Lemma}
\newtheorem{_theorem}{Theorem}
\newtheorem{_col}{Collorally}

\begin{document}

\begin{Large}
\centerline {
Sums with convolution of Dirichlet characters
}
\centerline {
Dmitry Ushanov
}
\end{Large}

\section{Introduction }

Let $\chi_1$ and $\chi_2$
be two primitive Dirichlet characters with conductors
$q_1$ and $q_2$ respectively.
In a recent paper [1] Banks and Shparlinski considered the sum
$S_{\chi_1,\chi_2}(T) = \sum_{0 < xy \leqslant T}\chi_1(x)\chi_2(y).$
For this sum they established upper bound
$$
S_{\chi_1,\chi_2}(T) \ll T^{13/18}q_1^{2/27}q_2^{1/9+o(1)}
$$
for $T\geqslant q_2^{2/3}\geqslant q_1^{2/3},$
and
$$
S_{\chi_1,\chi_2}(T) \ll T^{5/8}q_1^{3/32}q_2^{3/16+o(1)}
$$
for $T\geqslant q_2^{3/4}\geqslant q_1^{3/4}.$

In this paper we prove more precise
bounds on $S_{\chi_1,\chi_2}.$

\section{Statement of results }

\begin{_theorem}
Let $\chi_1$ and $\chi_2$ be two primitive
Dirichlet characters with conductors $q_1$ and $q_2$,
respectively. If $q_1 \leqslant q_2$ and $T > 1$
then for every $\epsilon>0$ one has
% r = 3
\begin{equation}
\label{main_result_1}
\sum_{0 < xy \leqslant T}\chi_1(x)\chi_2(y)
    \ll \left\{
\begin{array}{ll}
T^{2/3} (q_1 q_2)^{1/9 + \epsilon}
    & \textrm{if } (q_1q_2)^{1/3} \leqslant T\leqslant q_1^{4/3} q_2^{1/3}\textrm{,}\\
T^{3/4} q_2^{1/12 + \epsilon}
    & \textrm{if } q_1^{4/3}q_2^{1/3} \leqslant T,
\end{array} \right.
\end{equation}
% r = 2
\begin{equation}
\label{main_result_2}
\sum_{0 < xy \leqslant T}\chi_1(x)\chi_2(y)
    \ll \left\{
\begin{array}{ll}
T^{1/2} (q_1 q_2)^{3/16 + \epsilon}
    & \textrm{при } (q_1q_2)^{3/8} \leqslant T\leqslant q_1^{9/8} q_2^{3/8},\\
T^{2/3}q_2^{1/8+\epsilon}
    & \textrm{при } q_1^{9/8} q_2^{3/8} \leqslant T.
\end{array} \right.
\end{equation}

(Constant implied by $\ll$ depends only on $\epsilon$)
\end{_theorem}

\begin{_col}
Suppose that under the conditions of Theorem 1
we have $q_1 = q_2 = q.$
Then
$$
\sum_{0 < xy \leqslant T}\chi_1(x)\chi_2(y)
\ll \left\{
\begin{array}{ll}
T^{2/3} q^{2/9+\epsilon}
    & \textrm{if } q^{2/3} \leqslant T\leqslant q^{11/12},\\
T^{1/2} q^{3/8+\epsilon}
    & \textrm{if } q^{11/12} \leqslant T \leqslant q^{3/2},\\
T^{2/3} q^{1/8+\epsilon}
    & \textrm{if } q^{3/2} \leqslant T \leqslant q^{9/4},\\
T^{1/2} q^{1/2 + \epsilon}
    & \textrm{if } q^{9/4} \leqslant T.
\end{array} \right.
$$
\end{_col}

\textbf {Remark.}
In [1] under the conditions of Collorally 1
for $q^{2/3} \leqslant T \leqslant q^{83/84}$
it is shown that
$$
\sum_{0 < xy \leqslant T}\chi_1(x)\chi_2(y)
\ll T^{13/18} q^{5/27+o(1)}.
$$
Under this conditions our bound is more precise.

\begin{_theorem}
Let $\chi_1$ and $\chi_2$ be
two primitive Dirichlet
characters
with prime conductors
$q_1$ and $q_2$ respectively,
$q_1 \leqslant q_2,$
$T > 1$ and let
$r\geqslant 2$ be an integer.
Put
$$
\nu_r :=
    \left\{
\begin{array}{ll}
    1 & \textrm{if } r = 2,\\
    0 & \textrm{otherwise.}
\end{array} \right.
$$
Set
$$
T_r :=
    q_1^{\frac{(r+1)^2}{4r}}
    q_2^{\frac{r+1}{4r}}
    (\log q_1)^{r+1}
    (\log q_2)^{\nu_r r (r+1) + r^2 + 1}.
$$
Then
$$
\sum_{0 < xy \leqslant T}\chi_1(x)\chi_2(y)
    \ll \left\{
\begin{array}{ll}
T^{1-\frac 1 r} (q_1q_2)^{\frac{r+1}{4r^2}}
    \log^{\frac 1 r}q_1\log^{\frac 1 r+\nu_r + 1}q_2
    & \textrm{if } (q_1q_2)^{\frac{r+1}{4r}} \leqslant T\leqslant T_r,\\
T^{\frac{r}{r+1}} q_2^{\frac 1 {4r}}
    (\log q_2)^{\frac{2}{r+1}}
    & \textrm{if } T_r \leqslant T,
\end{array} \right.
$$
\end{_theorem}

\section{Basic notations}
Let $\chi_1$ and $\chi_2$ be two primitive Dirichlet characters
with conductors $q_1$ and $q_2.$
Suppose that $q_1 \leqslant q_2.$
Set $$Q := q_1q_2.$$
For a parameter $T>0$ we consider a hyperbola
$$\Gamma := \{(x,y)\in \mathbb{R}^2_+ \mid xy = T\}.$$
For a subset $\Omega \subset \mathbb{R}^2$
we define the character sum
$$S(\Omega) := \sum_{(x,y)\in\Omega \cap \mathbb{Z}^2}\chi_1(x)\chi_2(y).$$
For $k\in\mathbb{N}$ we define the value
$$\sigma_k := \frac 1 2 + \frac 1 4 + \cdots + \frac 1 {2^k}
= 1 - \frac 1 {2^k}.$$

In our proofs we will use the following result (see [2]).

\begin{_theorem}[Burgess]
For any primitive Dirichlet character $\chi$
of conductor q and any nonnegative
integers M, N we have
$$
\left | \sum_{M < n \leqslant M + N} \chi(n) \right |
    \leqslant c_\epsilon N^{1-\frac 1 r} q^{\frac{r+1}{4r^2} + \epsilon},
$$
where $r\in \{1,2,3\}.$
If $q$ is a prime number then
$$
\left | \sum_{M < n \leqslant M + N} \chi(n) \right |
    \leqslant c'_\epsilon N^{1-\frac 1 r} q^{\frac{r+1}{4r^2}} (\ln{q})^{\frac{1}{r}}
$$
for every $r\geqslant 1.$
\end{_theorem}

Suppose $I_1$ and $I_2$
are two intervals.
Then we define a rectangle
$I_1 \times I_2$
as follows:
$$
I_1 \times I_2 :=
    \{(x,y)\in\mathbb{R}^2\mid x\in I_1, y\in I_2\}.
$$
We write $|\Pi|$ for area of rectangle $\Pi$
and $\delta(\Pi) = length(I_1)$ for its width.

Consider rectangles
\begin{equation}
\label{U_0_def}
U_0 := (0;\sqrt{T})\times[0;\sqrt{T}),
\end{equation}
$$U_k := \left(0;\frac{\sqrt{T}}{2^k}\right)\times[2^{k-1}\sqrt{T};2^k\sqrt{T}),$$
where $k = 1,2,\ldots$
All rectangles $U_k$ have
one vertex on $\Gamma.$

Suppose that the rectangle
$$
\Pi = [x_0, x_1) \times [y_0, y_1)
$$
has a vertex $(x_1, y_1)$ on hyperbola $\Gamma,$
i.e. $x_1 y_1 = T.$
Then we define two new rectangles
$r(\Pi)$ and $u(\Pi)$ by the following rule:
$$
r(\Pi) :=
 \left[x_1, \frac{3x_1 - x_0}{2}\right)
 \times \left[y_0, \frac{2T}{3x_1-x_0}\right),
$$
$$
u(\Pi) :=
 \left[x_0,\frac{x_0+x_1}{2}\right)
 \times \left[y_1,\frac{2T}{x_0+x_1}\right).
$$

For $k=1,2,3,\ldots.$
we define rectangles

\begin{equation}
\label{Pi_k_def}
\Pi_k := \left[\frac{\sqrt{T}}{2^k}; \frac{3\sqrt{T}}{2^{k+1}}\right)
    \times\left[2^{k-1}\sqrt{T};\frac{2^{k+1}}{3}\sqrt{T}\right) = r(U_k).
\end{equation}

Define the set ${\cal F}_k$ of all rectangles that can be representented in the form
$
\sigma_1 \cdots \sigma_n \Pi_k,
$
where $\sigma_i \in \{r,u\}, i = 1,\ldots,n,$
for some $n\geqslant 0.$

If rectangle $\Pi\in {\cal F}_k$
is represented in the form
$\Pi = \sigma_1 \cdots \sigma_l \Pi_k $
then we say that $\Pi$
is a rectangle of order $l.$

\section{Lemmata}

\begin{lemma}\label{area}
Consider rectangle $\Pi.$
Suppose $P := |\Pi|$
and rectangle's height and width
are both greater than 1.
Then for every real $\epsilon>0$ one has
$$
|S( \Pi )|
    \ll \left\{
\begin{array}{ll}
P^{2/3} Q^{1/9 + \epsilon}\textrm{,}\\
P^{1/2} Q^{3/16 + \epsilon}.
\end{array} \right.
$$
\end{lemma}

\begin{proof}
It is sufficient to
apply Burgess' theorem with $r=3$ in first case
and with $r=2$ in second.
\end{proof}

\begin{lemma}\label{main}
Suppose
$1\leqslant x_0<x,$ $y := T/x.$ Put
$$\Delta := x - x_0,$$
$$x_n := x_0 + \Delta \sigma_n
= x - \frac{\Delta}{2^n},$$
$n=1,2,3,\dots$
Consider rectangles of the form
$$
\Phi_n = [ x_{n-1}, x_{n} ) \times
    [ y, T / x_{n} ).
$$
Then
$$
|\Phi_{n}|
= \frac{\Delta^2}{2^{2n}}\frac{y}{x_n},
\ \ \
|u(\Phi_{n})|
= \frac{|\Phi_{n}|} {4\cdot\left(1 - \frac{3}{2}\cdot\frac{\Delta}{x}\cdot\frac{1}{2^n} \right)}.
$$
\end{lemma}

\begin{proof}
Let $y_n = \frac T {x_n}$ then
$$y_n = \frac T {x_0 + \Delta \sigma_n}. $$
The area of rectangle $\Phi_n$ is equal to
\begin{eqnarray*}
|\Phi_n| = (x_n - x_{n-1})(y_n - y)
=\Delta(\sigma_n-\sigma_{n-1})(\frac T{x_n}-y)\\
= \frac \Delta {2^n} \frac{T-y(x-\Delta/2^n)}{x_0+\Delta\sigma_n}
= \frac{\Delta^2}{2^{2n}}\frac{y}{x_n}.
\end{eqnarray*}
So the first equality is proved.

Using the same argument we obtain
$$
|u(\Phi_n)| = \frac{{\Delta'}^2}{4}\frac{y_n}{x_n-\Delta'/2},
$$
where $\Delta' = \Delta / 2^n.$
Therefore
$$
|u(\Phi_n)| = \frac{\Delta^2}{4\cdot 2^{2n}}
\cdot\frac{T}{x_n(x_n-\Delta / 2^{n+1})}.
$$
Hence
$$
\frac{|\Phi_n|}{|u(\Phi_n)|}
= 4 \cdot \frac{y}{x_n} \cdot \frac{x_n}{T}
\cdot(x_n - \Delta / 2^{n+1})
= 4\cdot\frac{1}{x}\left(1-\frac{\Delta}{2^n}
- \frac{\Delta}{2^{n+1}}\right).
$$
\end{proof}

\begin{lemma}
Consider rectangle $\Pi \in {\cal F}_k.$
Then
$$
|r(\Pi)| \leqslant |\Pi| / 4.
$$
\end{lemma}

\begin{proof}
Suppose that under conditions of Lemma \ref{main}
$$\Pi = \Phi_1, \ \ r(\Pi) = \Phi_2.$$
Then
$$
\frac{|\Pi|}{|r(\Pi)|} = 4 \frac {x_2} {x_1}.
$$
But $x_2 > x_1$ so we obtain Lemma.
\end{proof}

\begin{lemma}
Consider rectangle $\Pi \in {\cal F}_k$
with vertex
$(x,y)$ on the hyperbola $\Gamma,$
so $xy = T.$
Let
$\delta=\delta(\Pi).$
Then
$$|u(\Pi)|\leqslant
\frac
    {|\Pi|}{4\left(1 - \frac{3\delta}{2x}\right)}.$$
\end{lemma}

\begin{proof}
Without loss of generality we can assume that
$\Phi$ is the first rectangle in
the sequence of rectangles from Lemma \ref{main}.

Then
$$
\frac{\Delta'}{x'} = \frac{2\delta}{x + \delta}
= \frac{\delta}{x}\frac{2}{1+\delta / x}
\leqslant \frac{2\delta}{x}.
$$
Therefore
$$
4\cdot\left(1 - \frac{3}{2}\cdot\frac{\Delta'}{x'}\cdot\frac{1}{2} \right)
\geqslant
    4\cdot\left(1 - \frac{3\delta}{2x} \right).
$$
\end{proof}

\begin{lemma}\label{area_estimate}
Let
$\Pi = \sigma_1 \cdots \sigma_l \Pi_k.$
Then
$$
|\Pi|
\leqslant
\frac{|\Pi_k|}{4^l\prod_{j=1}^{l}( 1 - \frac{3}{2} (\frac 2 3)^l )}.
$$
\end{lemma}

\begin{proof}

We will show that the ratio $\delta/x$
is reduced by a factor $\geq 3/2$ every time
when rectangle $\Pi$ is replaced by $u(\Pi)$ or $r(\Pi).$

Case 1.
Consider rectangle $\Pi$ with parameters $(\delta, x)$
and rectangle $u(\Pi)$ with parameters $(\delta',x').$
Then $\delta' = \delta / 2$ and $x' = x - \delta / 2.$
Therefore
$$
\frac{\delta'}{x'} = \frac{\delta/2}{x - \delta / 2}
= \frac{\delta}{2x}\frac{ 1 } { 1 - \frac{\delta}{2x}} \leqslant
\frac{2\delta}{3x},
$$
because $\delta/x \leqslant 1/2$ for
all rectangles in ${\cal F}_k.$

Case 2.
Consider rectangle $\Pi$ with parameters $(\delta, x)$
and rectangle $r(\Pi)$ with parameters $(\delta',x').$
Then $\delta' = \delta / 2$ and $x' = x + \delta / 2,$
therefore
$$
\frac{\delta'}{x'} = \frac{\delta/2}{x + \delta / 2}
= \frac{\delta}{2x}\frac{ 1 } { 1 + \frac{\delta}{2x}} \leqslant
\frac{\delta}{2x} \leqslant \frac{2\delta}{3x}.
$$

Lemma is proved by applying Lemma 3 and Lemma 4.
\end{proof}

\begin{lemma}\label{ksi}
Let real $\delta,$ $t$ and $T$ be such that
$1/2 < \delta < 1$ and $1<t<T^{\delta}.$
Set
$$\Xi_t := \{(x,y)\in\mathbb{R}^2_+\mid T - 2t \leqslant xy \leqslant T \}.$$
Then $\# (\Xi_t \cap \mathbb{Z}^2)\ll t \ln{T}.$
\end{lemma}

\begin{proof}
Number of integer points under hyperbola
can be estimated by
$$
\sum_{x=1}^{T}\left[\frac{T}{x}\right]
= T\ln{T} + (2\gamma -1)T + O(T^{1/2}).
$$
Therefore
$$
\# (\Xi_t \cap \mathbb{Z}^2)
= T\ln{T} + (2\gamma-1)T
- (T-2t)\ln(T(1-\frac{2t}{T})) - ( 2\gamma - 1 ) (T-2t)
+O(T^{1/2}).
$$
Thus, Lemma is proved.
\end{proof}

%%%%%%%%%%%%%%%%%%%%%%%%%%%%%%%%%%%%%%%%%%%%%%%%%%%%%%%%%%%%%%%%%%%%%%%%

\section{ Proof of Theorem 1 for small T}

Set
$\Omega = \{(x,y)\in \mathbb{R}_+^2\mid xy < T\},$
and
$\Omega_1 = \{(x,y)\in \mathbb{R}_+^2\mid xy < T, x < \sqrt{T} \}.$
Without loss of generality we can estimate only $S(\Omega_1).$

It is obviously that rectangles from ${\cal F}_k, k = 1,2,\ldots$
together with $U_k, k=0,1,2,\ldots$ cover
all the set $\Omega_1.$

Consider $t := T^{3/4}q_2^{1/12}$
and real $\eta>0.$ Set
\begin{equation}
\label{def_W_t}
W_t := \{(x,y)\in\mathbb{R}^2_+\mid xy < T,\ y \leqslant t,\ x \leqslant \sqrt{T} \},
\end{equation}
\begin{equation}
\label{def_W'_t}
W'_t := \{(x,y)\in\mathbb{R}^2_+\mid xy < T,\ y \geqslant t \},
\end{equation}
\begin{equation}
\label{def_Xi_t}
\Xi_t = \{(x,y)\in\mathbb{R}^2_+\mid T - 2t \leqslant xy \leqslant T \}.
\end{equation}

The number of integer points in $\Xi_t$
is bounded by $\#(\Xi\cap\mathbb{Z}^2)\ll
t Q^{\eta}.$

Consider a rectangle $\Pi \in {\cal F}_k$
such that $\Pi\subset W$ and
let $(x_0,y_0)$ be its left bottom vertex.
Set $\delta = \delta(\Pi).$

Let
$$ 2\delta = \frac{T}{y_0} - x_0 \leqslant 2. $$
Then $\Pi \subset \Xi_t.$
Indeed
$$
T - x_0 y_0 \leqslant 2 y_0 \leqslant 2t,
$$
therefore $T - 2t \leqslant x_0 y_0.$

Now we estimate $S(W_t).$

For rectangles $\Pi\in {\cal F}_k,$ $\Pi\subset W_t$ with
$\delta(\Pi) \geqslant 1$ we apply Lemma \ref{area}. All other
rectangles are lying in $\Xi_t.$

The sum $S(\Xi_t)$ is trivially bounded by the number of integer points in $\Xi_t.$

The number of rectangles $\Pi$ of order
$l$ is equal to $2^l.$ The area of such a rectangle $\Pi$
is bounded by $|\Pi| \ll |\Pi_k| / 4^l.$

Thus, we have the following bound for the character sum over all rectangles $\Pi$ of order $l$
and with $\delta(\Pi)\geqslant 1$:
$$
S(\Pi^{l}) \ll T^{2/3} Q^{1/9+\eta} 4^{-2l/3} 2^{l}.
$$

As the sum $\sum_{l=0}^{\infty}4^{-2l/3} 2^{l}$ converges
we see that the character sum over
all rectangles $\Pi$ with $\delta(\Pi) \geqslant 1$
is bounded by $\ll T^{2/3} Q^{1/9+\eta}.$

Therefore
\begin{equation}
\label{s_w_r_3}
S(W_t) \ll \max( T^{2/3} Q^{1/9+\eta}, t Q^{\eta} ).
\end{equation}

In order to estimate $S(W'_t)$
we use Burgess' Lemma with $r = 3:$

$$
S(W'_t) \ll \sum_{x=1}^{T/t}(T/x)^{2/3}q_2^{1/9+\eta}
\ll T^{2/3}(T/t)^{1/3}q_2^{1/9+\eta},
$$
therefore
\begin{equation}
\label{s_w'_r_3}
S(W'_t) \ll T t^{-1/3} q_2^{1/9+\eta}.
\end{equation}

So
$$
S(\Omega_1)
\ll \max(
 T^{2/3} Q^{1/9+\eta},
 t Q^{\eta},
 T t^{-1/3} q_2^{1/9+\eta}
).
$$

Using the definition of parameter $t$ we obtain the following result.
If $T \leqslant q_1^{4/3} q_2^{1/3}$
then $S(\Omega) \ll T^{2/3}Q^{1/9 + \eta}.$
If $T \geqslant q_1^{4/3} q_2^{1/3}$ then
$S(\Omega) \ll tQ^{\eta} = T^{3/4}q_2^{1/12+\eta}.$

%%%%%%%%%%%%%%%%%%%%%%%%%%%%%%%%%%%%%%%%%%%%%%%%%%%%%%%%%%%%%%%%%%%%%%%%%%%

\section{ Proof of Theorem 1 for large T}

Set
\begin{equation}
    \label{t_def_large_T} t:=T^{2/3}q_2^{1/8}.
\end{equation}

As before, we use the sets
$W_t,$ $W'_t$ and $\Xi_t$
defined in (\ref{def_W_t}),
(\ref{def_W'_t}) and (\ref{def_Xi_t}).

The only difference between this case and previous one is the
convergence argument. The sum over all rectangles of order $l$ can
be estimated by $T^{1/2} Q^{3/16+\eta/2}.$ Therefore the sum
$$
    \sum_{l=0}^{\infty} T^{1/2} Q^{3/16+\eta/2}
$$
does not converge.
But it is easy to see that if $l \gg \log T$
then every rectangle $\Pi$ of order $l$ lies in the set $\Xi_t.$
So we have
\begin{equation}
\label{s_w_r_2}
S(W_t) \ll \max( T^{1/2} Q^{3/16+\eta}, t Q^{\eta} ).
\end{equation}

Applying Burgess' Lemma with $r=2$
we obtain
$$
S(W'_t) \ll \sum_{x=1}^{T/t}(T/x)^{1/2}q_2^{3/16+\eta}
\ll T^{1/2}(T/t)^{1/2}q_2^{3/16+\eta},
$$
therefore
\begin{equation}
\label{s_w'_r_2}
S(W'_t) \ll T t^{-1/2} q_2^{3/16+\eta}.
\end{equation}

Inserting (\ref{t_def_large_T}) into (\ref{s_w_r_2}) and
(\ref{s_w'_r_2}), we have
if $T \leqslant q_1^{9/8} q_2^{3/8}$
then $S(\Omega) \ll T^{1/2} Q^{3/16+\eta},$
and if $T \geqslant q_1^{9/8} q_2^{3/8}$
then $S(\Omega) \ll t Q^{\eta} = T^{2/3}q_2^{1/8+\eta}.$

\section{ Prime moduli }

Set $r\geqslant 2$ and
$
t:= T^{\frac{r}{r+1}} q_2^{\frac{1}{4r}}
    \log^{\frac{1-r}{r+1}}q_2.
$
We use sets defined by
(\ref{def_W_t}), (\ref{def_W'_t}) and (\ref{def_Xi_t}) again.

Our argument to estimate $S(W_t)$ is similar.
For $k\geqslant 1$ there exist $2^l$
rectangles of order $l.$
Applying Lemma \ref{area_estimate}, we have that
the character sum over all rectangles of order $l$
is bounded by
$$
S(\Pi^{l}) \ll T^{1-\frac 1 r} Q^{\frac{r+1}{4r^2}}
    (\log q_1)^{\frac 1 r} (\log q_2)^{\frac 1 r}
    4^{-(1-\frac 1 r)l} 2^{l}.
$$

We consider two cases.

Case 1 ($r\geqslant 3$). The sum
$\sum_{l=0}^{\infty}4^{-(1-\frac 1 r)l} 2^{l}$
converges, so
$$
\sum_{l=0}^{\infty}
 S(\Pi^{l}) \ll T^{1-\frac 1 r} Q^{\frac{r+1}{4r^2}}
 (\log q_1)^{\frac 1 r} (\log q_2)^{\frac 1 r}.
$$

Case 2 ($r = 2$). The sum
$\sum_{l=0}^{\infty}4^{-(1-\frac 1 r)l} 2^{l}$
does not converge.
In this case it is sufficient
to take only first $\ll \log(T) \ll \log(q_2)$
values of $l.$

There are only $\ll \log T \ll \log q_2$ rectangles from ${\cal F}_k$
lying lower the line $y = t.$ So
$$
S(W'_t) \ll \max(
    T^{1-\frac 1r} Q^{\frac{r+1}{4r^2}}
        \log^{1/r}q_1\log^{\frac 1r+\nu_r + 1}q_2,
    t\log{q_2} ).
$$

By Burgess' Lemma, we have
$$
S(W'_t)\ll
\sum_{x=1}^{T/t} (T/x)^{1-\frac 1r}q_2^{\frac{r+1}{4r^2}}
(\log q_2)^{\frac 1r} \ll
T t^{-\frac 1r}q_2^{\frac{r+1}{4r^2}}
(\log q_2)^{\frac 1r}.
$$

\section {Proof of Collorally 1}

First three inequalities are
immediate consequences of Theorem 1.

We obtain the last inequality by applying Burgess' Lemma with
$r=1.$ We begin with splitting the sum over points under hyperbola
into three parts:
$$
\Omega_1 = \{(x,y)\in\mathbb{R}^2_+\mid xy < T, x < \sqrt{T}\},
$$
$$
\Omega_2 = \{(x,y)\in\mathbb{R}^2_+\mid xy < T, y < \sqrt{T}\},
$$
and $U_0,$ defined by (\ref{U_0_def}).

Applying Burgess' Lemma with $r = 1,$ we have
$$
S(\Omega_1) \ll \sqrt{T} q^{1/2 + \epsilon},
$$
and
$$
S(U_0) \ll q^{1+\epsilon}.
$$

We are interested in the case
$T\geqslant q^{3/2}.$
So $\sqrt{T} q^{1/2 + \epsilon} \geqslant q^{1+\epsilon}$
and we obtain the collorally.

\end{document}